\documentclass{amsart}

\title{Extremely amenable groups via continuous logic}
\author{Julien Melleray and Todor Tsankov}
\usepackage{enumerate}
\usepackage{amssymb}
\usepackage[initials,shortalphabetic]{amsrefs}
\usepackage[all]{xy}
\usepackage{hyperref}
\usepackage{mathpazo}

\usepackage{my-macros}

\begin{document}

\begin{abstract}
We establish a characterization of extreme amenability of any Polish group $G$ in Fra\"iss\'e-theoretic terms in the setting of continuous logic, mirroring a theorem due to Kechris, Pestov and Todorcevic for subgroups of $S_{\infty}$. 
\end{abstract}
\maketitle
\section{Introduction}

These notes are concerned with the theory of \emph{extremely amenable} Polish groups, i.e.~Polish groups $G$ with the following property: whenever $G$ acts continuously on a 
compact metric space $X$, there exists $x \in X$ such that $g \cdot x=x$ for all $g \in G$.

The question of the existence of such topological groups was asked in the late sixties, and the first example was published in 1975 by Christensen and Herer. The group in question was a so-called ``exotic'' Polish group, and it was initially thought that extreme amenability itself was a pathological property, especially since it 
was noticed early on that there exists no locally compact extremely amenable topological group. Since then, however, numerous ``natural'' Polish groups have been proved 
to be extremely amenable; examples include the unitary group of an infinite-dimensional separable Hilbert space (Gromov-Milman~\cite{Gromov1983}), the automorphism group of a standard probability space (Giordano-Pestov~\cite{Giordano2007}), and the isometry group of the universal Urysohn metric space (Pestov \cite{Pestov2002}). For a detailed discussion of the theory of extremely amenable 
Polish groups, as well as detailed bibliographical references, we refer the reader to Pestov's book \cite{Pestov2006a}, from which the quick historical discussion above was taken.

In the seminal article \cite{Kechris2005}, Kechris, Pestov and Todorcevic provided a characterization of extremely amenable subgroups of $S_{\infty}$, the permutation 
group of an infinite countable set. These groups may naturally be seen as automorphism groups of countable first-order structures, and the characterization obtained in 
\cite{Kechris2005} is in terms of a combinatorial property of Fra\"iss\'e classes, the so-called \emph{Ramsey property}. This characterization, which completely captures 
the combinatorial content of extreme amenability for subgroups of $S_{\infty}$, has led to some new examples of extremely amenable Polish groups, and one may also hope to use this connection between topological dynamics and combinatorics to obtain some new Ramsey-type theorems. However, the characterization of \cite{Kechris2005} cannot be applied to general Polish groups. 

In this article, we use the framework of 
\emph{continuous logic}, which was introduced in \cite{BenYaacov2010a}, and a notion of Fra\"iss\'e class adapted to that context -which we describe in detail in the next section- to provide a characterization of extreme amenability for Polish groups in terms of an ''approximate Ramsey property'' (which, in the discrete context, boils down to the characterization of \cite{Kechris2005}). Fra\"iss\'e classes were first studied in the setting of continuous logic in \cite{Schoretsanitis2007}; the more streamlined and efficient approach we follow here is due to Ben Yaacov \cite{BenYaacov2014}.

A very similar characterization of extreme amenability, in the more 
restrictive context of isometry groups of ultrahomogeneous Polish metric spaces, had been obtained earlier by Pestov in \cite{Pestov2002}. We then describe a method 
that may be used to prove that a metric Fra\"iss\'e class satisfies the approximate Ramsey property. This method is based on the concentration of measure phenomenon, and is at heart just a formalization of Pestov's ideas as presented in his book \cite{Pestov2006a}. 
Our method applies, for instance, to recover the (well-known) results that the isometry group of Urysohn's universal metric space and the automorphism group of a standard probability space are extremely amenable. It is unclear whether our approach can be used to obtain interesting new examples of extremely amenable Polish groups.\\

\noindent \emph{Acknowledgments.} We thank Ita\"i Ben Yaacov for numerous discussions about metric Fra\"iss\'e classes; the exposition of the basics of metric Fra\"iss\'e theory in the next section is based on the content of those discussions and his research notes \cite{BenYaacov2014}. Work on this article was initiated during the authors' visit of the Fields Institute in Toronto in October--November 2010 to attend the thematic program on Asymptotic Geometric Analysis organized by V. Milman, V. Pestov and N. Tomczak--Jaegerman. We are very grateful for the excellent working environment that we enjoyed during this visit. Work on this project was partially funded by ANR projects AGORA and GRUPOLOCO.

\section{Metric Fra\"iss\'e classes}

In this section we explain how Fra\"iss\'e classes are defined in the metric 
context. These classes were first considered in \cite{Schoretsanitis2007}, but the approach we follow is entirely borrowed from Ben Yaacov's work \cite{BenYaacov2014}. We adopt a slightly less general point of view than in \cite{BenYaacov2014}, which nevertheless seems to be sufficient to cover all natural examples.

\begin{defn}
A \emph{language} $\mcL$ is a set of \emph{relational symbols} $(R_i)_{i \in I}$ and \emph{functional symbols} $(f_j)_{j \in J}$. To each symbol of the language, two numbers are attached: its \emph{arity} (a natural number) and its \emph{Lipschitz constant} (a nonnegative real number). We allow no $0$-ary relations, while $0$-ary functions are allowed (and must be understood as constants).

A $\mcL$-structure $\bA$ consists of a \emph{complete} metric space $(A,d)$, along with:
\begin{itemize}
\item A function $R^{\bA} \colon A^n \to \R$ for each $n$-ary relational symbol of $\mcL$.
\item A function $f^{\bA} \colon A^n \to A$ for each $n$-ary functional symbol of $\mcL$.
\end{itemize}
Interpretations of both relational and functional symbols are additionally required to be Lipschitz, the Lipschitz constant being (at worst) the one specified in the language. For this to make sense, we must say that in this section we always endow products of metric spaces with the sup-metric: for $\bar a,\bar b \in A^n$, we set 
$d(\bar a, \bar b)= \max d(a_i,b_i)$.

Finally, if the underlying metric space $(A,d)$ of a metric structure $\bA$ is separable, we say that $\bA$ is a \emph{Polish} metric structure.
\end{defn}

We always assume that our languages contain a special binary symbol, which is interpreted by the distance function in all structures. This is analoguous to the way one usually treats the equality symbol in classical first-order logic and, indeed, the distance symbol is to be thought of as measuring how far two things are from being equal. When the distance is discrete (i.e.~ takes only the values $0$ and $1$), one recovers the usual notions of first-order logic.

\begin{defn}
When $\bA$ is a $\mcL$-structure, and $\bar a \in A^n$, the \emph{substructure} generated by $\bar a$ is the smallest \emph{closed} subset $B$ of $A$ containing $\bar a$ and stable under interpretations of all functions in $\mcL$. We turn it into a $\mcL$-structure in the natural way, and denote by $\str a$ the structure generated by $\bar a$.

A \emph{morphism} from a $\mcL$-structure $\bA$ to another $\mcL$-structure $\bB$ is a map $\phi$ which commutes with interpretation of the symbols of $\mcL$, i.e.~such that:
\begin{itemize}
\item For any $n$-ary relational symbol $R$ of $\mcL$, and any $\bar a \in A^n$, one has $R^{\bA}(\bar a)=R^{\bB}(\phi ( \bar a))$.
\item For any $n$-ary functional symbol $f$ of $\mcL$, and any $\bar a \in A^n$, one has $\phi (f^{\bA}(\bar a))= f^{\bB}(\phi (\bar a))$.
\end{itemize}
Note that, since the distance function is always part of our languages, a morphism from $\bA$ to $\bB$ always induces an isometric embedding from $(A,d)$ to $(B,d)$.
\end{defn}

\begin{remark*}
Our conventions here are both more, and less, relaxed than the usual definitions of first-order logic: more, because we allow unbounded spaces and maps; less, because we only allow Lipschitz maps. The first part is justified by the fact that one does not need the compactness theorem of first-order logic in order to have a satisfactory analogue of Fra\"iss\'e theory; the second part is simply a matter of convenience, as it makes exposition  a bit simpler to follow and all our examples may be expressed using a Lipschitz language. The approach of \cite{BenYaacov2014} is more general than the one we adopt here, and everything in the present article could be made to work in that setting.
\end{remark*}

\begin{defn}
We say that a $\mcL$-structure $\bA$ is \emph{homogeneous} if it satisfies the following condition: whenever $\bar a \in A^n$ is a finite tuple, and $\phi \colon 
\str a \to \bA$ is a morphism, there exists for any $\varepsilon >0$ an automorphism $\psi$ of $\bA$ satisfying $d(\phi(\bar a), \psi(\bar a)) < \varepsilon$.
\end{defn}

In classical (discrete) first-order logic, countable homogeneous structures can be characterized by properties of the class of their finitely-generated substructures; this is the starting point of Fra\"iss\'e theory. It turns out that this extends to the continuous setting. 

\begin{defn}
The \emph{age} of an $\mcL$-structure $\bA$ is the class of all finitely generated $\mcL$-structures which embed in $\bA$. 
\end{defn}

Let us begin by listing two properties that must be satisfied by the age of any structure.

\begin{defn}
We say that a class $\mcK$ of finitely generated $\mcL$-structures satisfies:
\begin{itemize}
\item the \emph{hereditary property} (HP) if every finitely generated substructure of an element of $\mcK$ also belongs to $\mcK$.

\item the \emph{joint embedding property} (JEP) if every two members of $\mcK$ embed simultaneously in a third one. 
\end{itemize}
\end{defn}

Homogeneity of a structure imposes further conditions on its age.

\begin{defn}
We say that a class $\mcK$ of finitely generated $\mcL$-structures satisfies the \emph{near amalgamation property} (NAP) if, whenever $\bA, \bB_0, \bB_1$ are elements of $\mcK$, and $\phi_i \colon \bA \to \bB_i$ are morphisms, there exists for any finite tuple $\bar a \in A^n$ and any $\varepsilon >0$ some $\bC \in K$ and morphisms $\psi_i \colon \bB_i \to C$ such that $d(\psi_0 \phi_0(\bar a), \psi_1 \phi_1(\bar a)) < \varepsilon$. 
\end{defn}

The final condition one must impose on the age of a separable, homogeneous structure is the analogue of countability in the discrete setting. It splits in two parts: separability, and completeness. To state it we must first introduce a family of pseudo-metrics.

\begin{defn}
Let $\mcK$ be a class of finitely generated structures satisfying (HP), (JEP) and (NAP). Denote by $\mcK_n$ the class of all pairs $(\bar a, \bA)$ where $\bA \in \mcK$ and $\bar a \in A^n$ generates $\bA$ (using a convenient abuse of notation, in the following we will often write $\str a \in \mcK_n$, or even $\str a \in \mcK$).

Then define a pseudo-metric $d_n$ on $\mcK_ n$ by setting 
$$d_n(\str a, \str b)= \inf_{\varphi, \psi} d(\varphi(\bar a),\psi(\bar b))$$
where $\varphi, \psi$ range over all possible embeddings of $\str a, \str b$ into a common structure $\bC \in \mcK$.

Checking that $d_n$ is indeed a pseudo-metric is left to the reader; the fact that it is well-defined follows from JEP, while the triangle inequality follows from NAP.
\end{defn}

\begin{defn}
We say that a class of finitely generated $\mcL$-structures $\mcK$ satisfies the \emph{Polish property} (PP) if each $d_n$ is separable and complete.
\end{defn}

\begin{defn}
A class of finitely generated $\mcL$-structures $\mcK$ is said to be a \emph{Fra\"iss\'e class} if it satisfies HP, JEP, NAP and PP.
\end{defn}

We now have the following results, which in this generality are due to Ben Yaacov \cite{BenYaacov2014}.

\begin{theorem}[Ben Yaacov \cite{BenYaacov2014}]
For a Polish $\mcL$-structure $\bA$, the following are equivalent:
\begin{itemize}
\item $\bA$ is homogeneous.
\item The age of $\bA$ is a Fra\"iss\'e class.
\end{itemize}
\end{theorem}

\begin{theorem}[Ben Yaacov \cite{BenYaacov2014}]
Given a Fra\"iss\'e class $\mcK$ of finitely generated $\mcL$-structures, there is (up to isomorphism) a unique homogeneous $\mcL$-structure whose age is equal to $\mcK$. This structure is called the \emph{Fra\"iss\'e limit} of $\mcK$.
\end{theorem}

The relevance of the above considerations to the theory of Polish groups comes from the following observation (see \cite{Melleray2010}, Theorem 6): any Polish group is isomorphic to the automorphism group of some approximately ultrahomogeneous Polish metric structure (in a countable Lipschitz language)- i.e, given what we saw above, of a Fra\"iss\'e limit. 

To give a bit more detail, let $\bA$ be a Polish relational metric structure with universe $(A,d)$; we endow its automorphism group $G$ with the pointwise convergence topology, 
i.e.~the topology which has a basis of neighborhoods of $1$ given by 
$$\{g \in G \colon \forall a \in F \ d(g(a),a) < \varepsilon \}$$
where $\varepsilon >0$ and $F$ is a finite subset of $A$. This turns $G$ into a closed subgroup of the Polish group $\Iso(A)$, and so, when endowed with this topology, 
$G$ is a Polish group in its own right.

Whenever we consider the automorphism group of a Polish relational metric structure $\bA$ as a Polish group, we endow it with the topology described above 
(which we call the \emph{natural} Polish topology on $\text{Aut}(\bA)$). 

Now let us reformulate Theorem 6 of \cite{Melleray2010} using the language of Fra\"iss\'e limits.

\begin{prop}\label{PolishFraisse}
Let $G$ be a Polish group. Then there exists a countable language $\mcL$, and a Fra\"iss\'e class $\mcK$ of finitely generated $\mcL$-structures, such that $G$ is isomorphic as a topological group to the automorphism group of the Fra\"iss\'e limit of $\mcK$. 
\end{prop}

In theory, this means that whenever we state a general result about automorphism groups of metric Fra\"iss\'e limits, we are talking about all Polish groups. For 
practical purposes, this is far from true - for example, we know of no way to turn the homeomorphism group of the Hilbert cube into the automorphism group of a 
Fra\"iss\'e limit in such a way that combinatorial properties of the corresponding Fra\"iss\'e class may be understood. Also in theory, relational languages are enough; for practical purposes, it is much more convenient to allow functional symbols.

Let us finish this section by discussing a few examples.

\begin{itemize} 
\item The class of all finite metric spaces is a Fra\"iss\'e class, in the language containing only one binary relational symbol interpreted by the distance function in all structures. The Fra\"iss\'e limit of that class is Urysohn's universal metric space. One could also consider a bounded variant, allowing only metric spaces of diameter, say, less than $1$, and obtaining as a limit the Urysohn sphere.
\item The class of all finite probability algebras may be seen as a Fra\"iss\'e class in the language $\mcL_{prob}=(d,\vee, \wedge, {}^c{},\mathbold 0)$, with $\vee$, $\wedge$ being binary $2$-Lipschitz functional symbols which are interpreted by union and intersection respectively in all members of the class, and ${}^c{}$ being a unary $1$-Lipschitz functional symbol interpreted by complementation, while $\mathbold 0$ is a $0$-ary function always interpreted by the empty set. Of course, the measure of a set $A$ is then given by $d(A,\emptyset)$. The limit of this class is the measure algebra of a standard atomless probability space; the automorphism group of this limit is (isomorphic to) the group of measure-preserving bijections of the unit interval.

\item The class of all finite-dimensional normed vector spaces, in the language $(\|\cdot \|, \mathbold 0,(\cdot_{\lambda})_{\lambda \in \Q},+)$ is a Fra\"iss\'e class (the Lipschitz constant for multiplication by $\lambda$ being $|\lambda|$, and the Lipschitz constant for $+$ being $2$). The limit of this space is the Gurarij space, an object that was first built by Gurarij \cite{Gurarij1966}, and whose uniqueness was proved by Lusky \cite{Lusky1976}. Seeing this space as a Fra\"iss\'e limit was first done by Ben Yaacov \cite{BenYaacov2014}, following earlier work of Henson. 
\end{itemize}

In the first two examples above, the structures considered are exactly homogeneous, in the sense that any morphism defined on a finitely generated substructure extends to an isomorphism of the whole space; accordingly, both classes satisfy an exact amalgamation property, where one can take $\varepsilon=0$ in the definition of NAP. The third example is particulary interesting in our context, in that it shows that the $\varepsilon$'s in the definitions of this section are sometimes truly needed: it is well-known that there is no separable Banach space which is both universal for finite-dimensional normed spaces and exactly homogeneous (actually, any exactly homogeneous separable Banach space is a Hilbert space, while homogeneity must already fail for $1$-dimensional subspaces in any universal separable Banach space).

\section{Approximate Ramsey property}
In this section, $\mcK$ is a Fra\"iss\'e class of finitely generated structures, with limit $\bK$. We let $G$ denote the automorphism group of $\bK$, endowed with its natural Polish topology.
To give some intuition of what is going on, we use the same vocabulary as in the discrete setting (most notably, we will speak of colorings); part of what follows in this section was already done by Pestov \cite{Pestov2002}, in the context of ultrahomogeneous metric spaces, while in the discrete context our results are just those of \cite{Kechris2005}.

\begin{defn}
Given $\bA=\str a, \bB \in \mcK$, we denote by $\emb{A}{B}$ the set of embeddings of $\bA$ in $\bB$. We endow this set with the metric $\rho_{\bar a}$ defined by 
$$\rho_{\bar a}(\alpha,\beta)= d(\alpha(\bar a),\beta(\bar a)) \ . $$
A \emph{coloring} of $\emb{A}{B}$ is a $1$-Lipschitz map $\gamma \colon (\emb{A}{B},\rho_{\bar a}) \to [0,1]$.

Note that the metric $\rho_{\bar a}$, hence also the notion of a coloring, depends on the choice of generators for $\bA$. This is a cost one has to pay if one wants to allow functional symbols: one has to name the generators of a structure whenever one computes distances. 

We define similarly $\emb{A}{K}$, a metric (still denoted by $\rho_{\bar a}$) on $\emb{A}{K}$, and colorings of $\emb{A}{K}$. 
\end{defn}

\begin{defn}
Let $\bA \le \bB \in \mcK$, $\bC \in \mcK$ and $\beta \in \emb{B}{C}$. For $F$ a subset of $\emb{A}{B}$, we define a subset $F(\beta) \subseteq \emb{A}{C}$ by
$$\alpha \in F(\beta) \ \Leftrightarrow \ \exists \delta \in F \quad \alpha= \beta \circ \delta\ . $$ 
We define similarly $F(\beta)$ for $\beta \in \emb{B}{K}$ and $F$ a subset of $\emb{A}{B}$. 
\end{defn}

Here, the situation is complicated by the fact that we allow functional symbols: if we were in a purely relational setting, $\emb{A}{B}$ would be finite for all $\bA,\bB \in \mcK$; considering the case of finite-dimensional normed vector spaces shows that this is no longer true in the functional setting.

\begin{defn}
We say that $\mcK$ has the \emph{approximate Ramsey property} if the following happens: for all $\str a=\bA \le \bB \in \mcK$, all nonempty finite $F \in \emb{A}{B}$, and all $\varepsilon >0$, there exists $\bC \in \mcK$ such that for any coloring $\gamma$ of $\emb{A}{C}$ there exists $\beta \in \emb{B}{C}$ such that the oscillation of $\gamma$ restricted to $F(\beta)$ is less that $\varepsilon$.
\end{defn}

Let us point out that, for us, the empty function on the empty set is $1$-Lipschitz, hence the assumptions above imply that $\emb{B}{C}$ is nonempty; also, the oscillation above is of course computed according to the distance $\rho_{\bar a}$, hence depends on the choice of generators for $\bA$.

\begin{prop} \label{caracARP}
The following properties are equivalent:
\begin{enumerate}[(i)]
\item $\mcK$ has the approximate Ramsey property.
\item  For any finite substructures $\str a= \bA \le \bB$ of $\bK$, for any $\varepsilon >0$, any finite nonempty $F \subseteq \emb{A}{B}$ and any coloring $\gamma$ of $\emb{A}{K}$, there exists $\beta \in \emb{B}{K}$ such that the oscillation of $\gamma$ on $F(\beta)$ is less than $\varepsilon$.

\end{enumerate}
\end{prop}

\begin{proof}

$(i) \Rightarrow (ii)$: Pick finite substructures $\str a=\bA \le \bB \le \bK$, and fix $\varepsilon >0$. We may find $\bC \in \bK$ witnessing the approximate Ramsey property and assume that $\bC \le \bK$. Now let $\gamma$ be a coloring of $\emb{A}{K}$; $\gamma$ restricts to a coloring of $\emb{A}{C}$, hence we may find an embedding $\beta \colon \bA \to \bC$ such that the oscillation of $\gamma$ on $\emb{A}{C}(\beta)$
is less than $\varepsilon$. Using homogeneity of $\bK$, we may find an automorphism $\beta'$ of $\bK$ such that $d(\beta'(\alpha(\bar a)),\beta(\alpha(\bar a))) \le \varepsilon$ for all $\alpha \in F$; then the fact that $\gamma$ is $1$-Lipschitz, along with the triangle inequality, yields that the oscillation of $\gamma$ on $F(\beta')$ is less than $2 \varepsilon$. Since $\varepsilon$ was arbitrary, this concludes the proof.

\noindent$(ii) \Rightarrow (i)$: We proceed by contraposition, and assume that $(i)$ is false. 
Thus, we may find $\str a= \bA \le \bB \in \mcK$, a finite $F \subseteq \emb{A}{B}$ and $\varepsilon >0$ such that for any $\bC \in \mcK$ there exists a bad coloring $\gamma$ of $\emb{A}{C}$, 
i.e.~for any $\beta \in \emb{B}{C}$ the oscillation of $\gamma$ on $F(\beta)$ is larger than $\varepsilon$ 
(note that if there is no embedding from $\bB$ into $\bC$ the preceding property holds for any $\gamma$). 
For any finitely generated substructure $\bC$ of $\bK$, we pick such a bad coloring $\gamma_{\bC}$ of $\emb{A}{C}$.
  
We fix an ultrafilter $\mcU$ on the set of finite subsets of $\bK$ with the property that for any finite subset $D$ of $\bK$ one has 
$$\{E \subseteq \bK \colon E \text{ is finite and } D \subseteq E\} \in \mcU \ .$$
We define a mapping $\gamma \colon \emb{A}{K} \to [0,1]$ by setting $\gamma = \lim_{\mcU} \gamma_{\langle C \rangle}$. In other words, for any $\alpha \in \emb{A}{K}$, one has 
$$\gamma(\alpha)=t \Leftrightarrow \forall \delta >0 \ \{C \colon \gamma_{\langle C \rangle}(\alpha) \in [t-\delta,t+\delta]\} \in \mcU  \ .$$
Note that the set $\{C \colon \alpha(A) \subseteq C\}$ belongs to $\mathcal U$ so the above definition makes sense (for any such $C$, $\gamma_{\langle C \rangle}(\alpha)$ is well-defined, 
and the ultrafilter thinks that the set made up of all the other $C$'s is negligible). 
Also, since each $\gamma_{\langle C \rangle}$ is $1$-Lipschitz it is immediate that $\gamma$ is also $1$-Lipschitz, in other words $\gamma$ is a coloring of $\emb{A}{K}$. 
We simply have to check that $\gamma$ witnesses the fact that $\bK$ fails to have property $(ii)$.

To that end, pick $\beta \in \emb{B}{K}$ and let $\bar b$ be a finite tuple generating $\bB$. Then $U_{\beta}=\{D \colon \beta(\bar b) \subseteq D\}$ belongs to $\mcU$. Denoting by $\alpha_1,\ldots,\alpha_n$ the elements of $F(\beta)$, and using the fact that $\beta \in \emb{B}{D}$ for any $D \in U_{\beta}$, we know that for any $D \in U_{\beta}$ there exist 
$i,j \in \{1,\ldots,n\}$ such that 
$$\left | \gamma_{\langle D \rangle}(\alpha_i) - \gamma_{\langle D \rangle}(\alpha_j) \right| \ge \varepsilon\ . $$
Hence there exist $i,j \in \{1,\ldots,n\}$ such that 
$$\left \{D \colon \left | \gamma_{\langle D \rangle}(\alpha_i) - \gamma_{\langle D \rangle}(\alpha_j) \right| \ge \varepsilon \right \} \in \mcU \ .$$
Fixing such a pair $i,j$, we obtain that $|\gamma(\alpha_i)- \gamma(\alpha_j)| \ge \varepsilon$, so the oscillation of $\gamma$ on $F(\beta)$ is larger than $\varepsilon$. Since $\beta$ was arbitrary, this shows that $\bK$ does not have property $(ii)$.

\end{proof}

\begin{defn}  \label{eps-approx}
Let $\bar a \in K^n$, and $\alpha_1,\ldots, \alpha_n$ morphisms from $\bA = \str a$ to $\bK$. If $\varepsilon >0$, $\bB$ is a finitely generated substructure of $\bK$ containing $\bA$, and $\alpha_1',\ldots,\alpha_n'$ are partial morphisms from $\bB$ to $\bK$ with domain $\bA$, we say that $(\bB,\alpha_1',\ldots,\alpha_n')$ \emph{$\varepsilon$-approximates} $(\bA,\alpha_1,\ldots,\alpha_n)$ if for all $i$ one has $d(\alpha_i(\bar a),\alpha_i'(\bar a)) \le \varepsilon$.
\end{defn}

\begin{defn}
We say that $\mcK$ has the \emph{weak approximate Ramsey property} (WARP) if  for any $\bar a \in K^n$, any finitely generated substructure $\bB$ of $\bK$ containing $\bA= \str a$, any partial morphisms $\alpha_1,\ldots,\alpha_m$ of $\bB$ with domain $\bA$ and any $\varepsilon >0$, there exists a substructure $\bB'$ of $\bK$ containing $\bA$ and a finite set of partial morphisms $F=\{\alpha_1',\ldots,\alpha_m'\}$ of $\bB'$ with domain $A$ such that :
\begin{itemize}
\item $(\bB',\alpha_1',\ldots,\alpha_m')$ $\varepsilon$-approximates $(\bA,\alpha_1,\ldots,\alpha_m)$.
\item There exists a finitely generated structure $\bC \in \mcK$ such that, for any coloring $\gamma$ of $\emb{A}{C}$, there exists an embedding $\beta \colon \bB' \to \bC$ such that the oscillation of $\gamma$ on $F(\beta)$ is less than $\varepsilon$.
\end{itemize}
\end{defn}

We are now almost ready to state, and prove, the main result of this section, which is the characterization of extreme amenability 
of $G$ in terms of the approximate Ramsey property for $\mcK$ (extending Pestov's results from \cite{Pestov2002} to the context of metric Fra\"iss\'e classes). Before this, we need to recall a criterion for extreme amenability, and set our notations.

\begin{defn}
If $\bar a$ is a finite subset of elements of $\bK$, we let $d_{\bar a}$ denote the pseudometric on $G$ defined by 
$$\forall g, h \in G \ d_{\bar a}(g,h)= d(g(\bar a),h(\bar a))  \ . $$
\end{defn}

Slightly abusing notation, we will still denote by $(G,d_{\bar a})$ the metric space obtained by identyfing elements $g,h$ such that $d_{\bar a}(g,h)=0$.

The pseudometric $d_{\bar a}$ is obviously related to the metric $\rho_{\bar a}$ on $\emb{A}{K}$ - 
actually, it is almost the same thing, as witnessed by the following lemma.

\begin{lemma}
Let $\bA= \str a$ be a finitely generated substructure of $\bK$, and denote by $\Phi_{\bar a} \colon G \to \emb{A}{K}$ the mapping defined by
$\Phi_{\bar a}(g)=g_{|A} $.
Then $\Phi_A$ is a distance-preserving map from $(G,d_{\bar a})$ into $(\emb{A}{K},\rho_{\bar a})$ and $\Phi_{\bar a}(G)$ is dense in $\emb{A}{K}$.
\end{lemma}

\begin{proof}
The first part of the statement is obvious from the definitions of $d_{\bar a},\rho_{\bar a}$. 
The fact that $\Phi_A(G)$ is dense in $\emb{A}{K}$ is equivalent to saying that $\bK$ 
is homogeneous.
\end{proof}

In particular, any $1$-Lipschitz map $f \colon (G,d_{\bar a}) \to [0,1]$ uniquely extends to a coloring $\gamma_f$ of $\emb{A}{K}$ (by identifying $(G,d_{\bar a})$ and $\Phi_{\bar a}(G)$), 
while any coloring $\gamma$ of $\emb{A}{K}$ restricts to a $1$-Lipschitz map $f_{\gamma} \colon (G,d_{\bar a}) \to [0,1]$. 

Finally, we state a criterion of extreme amenability for a Polish group; this is essentially Theorem 2.1.11 in Pestov's book \cite{Pestov2006a}, reformulated to fit the fact that 
we consider $G$ as the automorphism group of a Fra\"iss\'e limit, and so we have a natural directed collection of left-invariant pseudometrics defining 
the topology of $G$: the metrics $d_{\bar a}$ we introduced above, where $\bar a$ ranges over all finite tuples of elements of the Fra\"iss\'e limit.

\begin{prop}\label{caracPestov}
$G$ is extremely amenable if, and only if, for any finite tuple $\bar a$ of $K$ the left-translation action of $G$ on $(G,d_{\bar a})$ is finitely oscillation stable, i.e:

For any finite subset $F$ of $G$, any $\varepsilon >0$ and any $1$-Lipschitz map $f \colon (G,d_{\bar a}) \to [0,1]$, 
there exists $g \in G$ such that the oscillation of $f$ on $gF$ is less than $\varepsilon$.

\end{prop}

\begin{theorem}\label{maincarac}
The following are equivalent, for a Fra\"iss\'e metric class $\mcK$ and $G$ the automorphism group of its limit, endowed with its natural Polish topology:
\begin{enumerate}[(i)]
\item $G$ is extremely amenable.
\item $\mcK$ has the approximate Ramsey property.
\item $\mcK$ has the weak approximate Ramsey property.
\end{enumerate}
\end{theorem}

\begin{proof}
$(i) \Rightarrow (ii)$: We assume that $G$ is extremely amenable, and we want to show that $\bK$ satisfies condition $(ii)$ of Proposition \ref{caracARP}.
To that end, we fix finitely generated substructures $\str a= \bA \le \bB$ of $\bK$ and $\varepsilon >0$, and consider a coloring $\gamma$ of $\emb{A}{K}$. We also let $\alpha_1,\ldots,\alpha_n$ enumerate the elements of some finite nonempty $F \subseteq\emb{A}{B}$, and use the approximate ultrahomogeneity of $\bK$ to find $g_1,\ldots,g_n \in G$ such that $\rho_{\bar a}(\alpha_i,{g_i}_{|A}) \le \varepsilon$. 

We set $F=\{g_1,\ldots,g_n\}$, and consider the $1$-Lipschitz function $f_{\gamma} \colon (G,d_{\bar a}) \to [0,1]$ induced by $\gamma$. By Proposition \ref{caracPestov}, 
we know that there exists $g \in G$ such that the oscillation of $f_{\gamma}$ on $gF$ is less than $\varepsilon$.

Note that we have for all $i \in \{1,\ldots,n\}$ that $\rho_{\bar a}(g{g_i}_{|A},g\alpha_i)=\rho_{\bar a}({g_i}_{|A},\alpha_i) \le \varepsilon$; the triangle inequality implies that, for any $i,j$:
$$\left|\gamma(g \alpha_i) - \gamma (g \alpha_j) \right| \le \left|\gamma(g\alpha_i)- \gamma(g{g_i}_{|A}) \right |
+\left|\gamma(g{g_i}_{|A}) - \gamma(g{g_j}_{|A}) \right|
+\left|\gamma(g{g_j}_{|A}) - \gamma (g \alpha_j) \right|$$

The first and third term above are both smaller than $\varepsilon$ since $\gamma$ is $1$-Lipschitz, and the term in the middle is $|f_{\gamma}(gg_i)-f_{\gamma}(gg_j)|$, which is less than $\varepsilon$ by definition of $g$. Denoting by $\beta$ the restriction of $g$ to $B$, we just proved that the oscillation of $\gamma$ on $\emb{A}{K}(\beta)$ is less than $3\varepsilon$, and this proves that $\mcK$ has the approximate Ramsey property.\\

\noindent $(ii) \Rightarrow (iii)$ is obvious. \\

\noindent $(iii) \Rightarrow (i)$: assume that $\mcK$ has the WARP, let $\bar a$ be a finite tuple of elements of $K$, $f \colon (G,d_{\bar a}) \to [0,1]$ be a $1$-Lipschitz map, and fix $\varepsilon >0$. Set $\bA= \str a$, and recall that $\gamma_f$ denotes the coloring of $\emb{A}{K}$ induced by $f$. We fix $g_1,\ldots,g_n \in G$ and assume without loss of generality that $g_1=id$. Define 
$\bB =\langle \bigcup_{i=1}^n g_i(A) \rangle$.

This is a finitely generated structure of $\mcK$ containing $\bA$ hence, using the WARP, we may find $\bB' \le \bK$ and embeddings $\alpha_1,\ldots,\alpha_n \colon \bA \to \bB'$ such that: 
\begin{itemize}
\item $d(\alpha_i(\bar a),g_i(\bar a)) \le \varepsilon$ for all $i$;
\item There exists a finitely generated $\bC \in \mcK$ such that for any coloring $\gamma$ of $\emb{A}{C}$ there exists $\beta \in \emb{B'}{C}$ such that the oscillation of $\gamma$ on $\{\beta \circ \alpha_1,\ldots,\beta \circ \alpha_n\}$ is less than $\varepsilon$. 
\end{itemize}

We may, and do, assume that $\bC \le \bK$, and apply the above property to $\gamma_f$; this yields an embedding $\beta \in \emb{B'}{C}$ such that $\gamma_f$ has oscillation less than $\varepsilon$ on $\{\beta \circ \alpha_1,\ldots,\beta \circ \alpha_n\}$. Using the homogeneity of $\bK$, we may find $g_{\beta} \in G$ such that
for all $i$ one has $d(g_{\beta} \circ \alpha_i(\bar a),\beta \circ \alpha_i(\bar a)) \le \varepsilon$.

Using the triangle inequality as in the proof of $(i) \Rightarrow (ii)$, we obtain, using straightforward computations, that for any 
$i,j \in \{1,\ldots,n\}$ one has 
$$\left|f(g_{\beta}g_i)-f(g_{\beta}g_j) \right| \le 7 \varepsilon \ .$$ 
Since $\varepsilon$ was arbitrary, this is enough to show that $G$ is extremely amenable.
\end{proof}

A very similar statement was obtained by Pestov in \cite{Pestov2002} in the case of ultrahomogeneous metric spaces. He was of course 
not using the vocabulary of metric Fra\"iss\'e classes, 
so the characterization he obtained was similar to point (ii) in our Proposition \ref{caracARP} characterizing the approximate Ramsey property.\\

We should probably point out here that there are (usually) many ways to turn a given Polish group $G$ into the automorphism group of 
the Fra\"iss\'e limit of some class $\mcK$; the properties of the class $\mcK$ may vary depending on the construction chosen 
(for instance, $\mcK$ may some times have the exact amalgamation property 
and sometimes only the near-amalgamation property - this happens e.g whenever $G$ is a subgroup of the permutation group of the integers whose topology does not 
admit a compatible complete left-invariant metric), 
however the theorem above implies that the approximate Ramsey property does not depend on the particular class $\mcK$, but 
only on whether $G$ is extremely amenable or not. This is not surprising, as it reflects a similar phenomenon uncovered for discrete Fra\"iss\'e classes in \cite{Kechris2005}.\\

\noindent {\bf An example.} By the result of Gromov and Milman mentionned in the introduction, the orthogonal group of a separable, infinite-dimensional (real) Hilbert space is extremely amenable, from which we deduce that the 
isometry group of the unit sphere of the Hilbert space is extremely amenable (these groups are one and the same). Since this unit sphere is the Fra\"iss\'e limit of 
the class of finite spherical metric spaces of diameter at most $2$, we obtain the approximate Ramsey property for this class. In turn, this easily implies the ARP for the 
class of all spherical metric spaces and, going back to the group side, we thus obtain that the isometry group of the unbounded universal spherical metric space is 
extremely amenable (see \cite{Pestov2006a}, exercise 5.1.32 p112 for the definition of this space, and Blumenthal \cite{Blumenthal1970} for informations about 
spherical metric spaces).

\section{A method to prove that certain Fra\"iss\'e classes have the approximate Ramsey property}

We again denote by $\mcK$ a metric Fra\"iss\'e class, by $\bK$ its Fra\"iss\'e limit, and by $G$ the automorphism group of $\bK$.

\begin{defn}
We say that $\mcK$ has the \emph{extension property} if for any $\bA \in \mcK$ there is $\bB \in \mcK$ such that $\bA$ embeds in $\bB$ and any partial automorphism of 
$\bA$ extends to an automorphism of $\bB$. 
\end{defn}

Using a deep theorem due to Hervig and Lascar \cite{Herwig2000}, Solecki proved in \cite{Solecki2005a} that the class of finite metric spaces has the extension property. 

In general, establishing that a class has this property seems very difficult. We will make use of a variant (which is weaker in the relational context) to show that certain classes have the approximate Ramsey property.

\begin{defn}
We say that $\mcK$ has the \emph{weak extension property} if for any $\bA = \str a \le \bK^n$ and any $\alpha_1,\ldots, \alpha_m \in \emb{\bA}{K}$ there exists a finitely generated substructure $\bB$ of $\bK$ and automorphisms $g_1,\ldots,g_n$ of $\bB$ such that:
\begin{itemize} 
\item $(\bB,g_1,\ldots,g_m)$ $\varepsilon$-approximates $(\bA,\alpha_1,\ldots,\alpha_m)$
\item $g_1,\ldots,g_m$ generate a relatively compact subgroup of $\Aut(\bB)$.
\end{itemize}
\end{defn}

The second condition may seem surprising; it is nevertheless essential for our purposes. This condition is empty whenever one only considers relational structures, or more generally when automorphism groups of elements of $\mcK$ are compact, which is the case in all our examples.

\begin{defn}
Whenever $(X,d)$ is a metric space, we let $d_n$ denote the normalized $\ell_1$-metric on $X^n$, defined by 
$$d_n\left((x_1,\ldots,x_n),(x'_1,\ldots,x'_n) \right)=\frac{1}{n}\sum_{j=1}^n d(x_i,x'_i)\ . $$ 
\end{defn}

\begin{defn}
We say that $\mcK$ is a \emph{$\ell_1$ metric Fra\"iss\'e class} (or that $\mcK$ \emph{has the $\ell_1$ property}) if for all $\bA \in \mcK$ and all $n$ one may turn $(A^n,d_n)$ 
into an element $\bA^n$ belonging to $\mcK$ in such a way that, whenever $(g_1,\ldots,g_n)$ are automorphisms of $\bA$, the diagonal map 
$$a \mapsto (g_1(a),\ldots,g_n(a)) $$
is an embedding of $\bA$ into $\bA^n$.
\end{defn}

For example, the class of all finite metric spaces is a metric Fra\"iss\'e class with both the $\ell_1$ property and the extension property. Other examples of classes with the $\ell_1$ property include 
the class $\mcK_X$ of finite metric spaces containing a common finite metric space $X$, the class of finite probability algebras, and the class of finite-dimensional normed vector spaces. A non-example would be, of course, the class of euclidean 
metric spaces.

The main result of this section is the following.

\begin{theorem}\label{l1-Ramsey}
Let $\mcK$ be a $\ell_1$ metric Fra\"iss\'e class with the weak extension property. Then $\mcK$ has the weak approximate Ramsey property, so $G=\Aut(Flim(\mcK))$ is extremely amenable.
\end{theorem}

\begin{proof}
Consider finitely generated structures $\bA=\str a \le \bB \in \mcK$, and fix $\varepsilon >0$. 

Let $\alpha_1,\ldots,\alpha_k$ enumerate a finite set of partial automorphisms of $\bB$ 
with domain $A$.
Applying the weak extension property, we know that we may find some finitely generated $\bB' \le \bK$ containing $\bA$
and automorphisms $g_1,\ldots,g_k$ of $B'$ such that $(\bB',g_1,\ldots,g_k)$ $\varepsilon$-approximate $(\bA,\alpha_1,\ldots,\alpha_k)$ 
and the closure $H$ of the subgroup of $\Aut(\bB')$ generated by $g_1,\ldots,g_k$ is  compact. 

Fix a set of generators $b_1,\ldots,b_q$ of $\bB'$ containing $g_i(\bar a)$ for all $i$, and endow $H$ with the bi-invariant metric $\delta$ defined by 
$$\delta(h_1,h_2) = \max \{d(h_1h (\bar b), h_2h (\bar b)) \colon h \in H \} \ . $$

For any $n$, we turn $((\bB')^n,d_n)$ into an element of $\mcK$ in a way that witnesses the $\ell_1$-property, and call this structure $\bC_n$. 
We want to show that for $n$ big enough, for any coloring $\gamma$ of $\emb{A}{C_n}$, there is $\beta \in \emb{B'}{C_n}$ such that the oscillation of $\gamma$ on 
$\{\beta \circ {g_1}_{|A},\ldots,\beta \circ {g_k}_{|A}\}$ is less than $\varepsilon$. 
To see this, we will use concentration of measure - this is where the $\ell_1$ property and our compactness assumption are useful. 

For any integer $n$, endow the group $H^n$ with the normalized $\ell_1$ metric $\delta_n$.
A coloring $\gamma$ of $\emb{A}{C_n}$ naturally induces a mapping, denoted by $\gamma'$, from $H^n \to  [0,1]$, since any $(h_1,\ldots,h_n)$ gives an embedding of $\bA$ into $C_n$. It follows from our definitions that $\gamma'$ is a $1$-Lipschitz map from $(H^n,\delta_n)$ to $[0,1]$.

Now, endow $H_n$ with its Haar measure $\mu_n$, and recall that the concentration of measure phenomenon for compact groups (see Theorem 4.3.19 in \cite{Pestov2006a}, or Theorem 4.2 in \cite{Ledoux2001}) ensures that for $n$ big enough we have, for any $1$-Lipschitz function 
$f \colon (H^n,\delta_n) \to [0,1]$, that 
\begin{equation} \tag{$\ast$} \label{ast}
\mu_n \left(\{\bar h \colon \left| f(\bar h)-E(f) \right| \le \varepsilon\} \right) > 1- \frac{1}{k} \ . 
\end{equation}
($E(f)$ denotes the expected value of $f$, i.e.~$\int f d\mu_n$)

Find such an $n$. We claim that $C_n$ has the desired property. 
To see this, we first define for all $i \in \{1,\ldots,k\}$ a measure-preserving bijection $\Theta_i$ from $H^n$ to itself,
obtained by setting $\Theta_i(h_1,\ldots,h_n)= (h_1 g_i,\ldots,h_n g_i)$.

Now, fix a coloring $\gamma$ of $\emb{A}{C_n}$ and let $\gamma'$ denote the corresponding mapping from $H^n$ to $[0,1]$.
Since each $\Theta_i$ preserves the measure $\mu_n$, and $\gamma'$ is $1$-Lipschitz, we may apply \eqref{ast} to see that there exists some 
$\bar h \in H^n$ such that 
$$\forall i \in \{1,\ldots,k\} \ \left|\gamma'(\Theta_i(\bar h)) - E(\gamma')  \right | \le \varepsilon$$
From this, it follows that 
$$\forall i,j \in \{1,\ldots,k\} \left|\gamma'(\Theta_i(\bar h))-\gamma'(\Theta_j(\bar h)) \right| \le 2 \varepsilon \ . $$
Going back to the definition of $\gamma'$, we have just obtained 
$$\forall i, j \in \{1,\ldots,k\} \ \left|\gamma((h_1,\ldots,h_n) \circ g_i) - \gamma(h_1,\ldots,h_n) \circ g_j)  \right| \le 2 
\varepsilon \ .$$
In other words, setting $\beta=(h_1,\ldots,h_n)$, we have shown that $\mcK$ has the WARP.

\end{proof}

Note that this enables one to recover the fact that the isometry group of the Urysohn space $\bU$ is extremely amenable. Also, 
for $X$ a finite metric space, the limit of $\mcK_X$ is $\bU$ endowed with a distinguished copy of $X$, whose automorphism group is exactly the (pointwise) 
stabilizer of $X$. So we see that the group of isometries that fix $X$ pointwise is extremely amenable for any finite subset $X$ of $\bU$; similarly for the group of measure-preserving bijections of a standard atomless probability space. 
Actually, in all the previous cases, Pestov's ideas, as reformulated in Theorem \ref{Levy} below, imply that the groups are L\'evy. 

It is relatively easy to use the theorem above in order to provide somewhat artificial new examples of extremely amenable groups; for instance, one can show that the class of all finite metric spaces endowed with a $1$-Lipschitz unary predicate is a Fra\"iss\'e class satisfying all our conditions above. A more interesting candidate for application of our approach is the isometry group of the Gurarij space $\bG$: certainly the class of finite-dimensional Banach spaces has the $\ell_1$ property; it is not clear whether this class satisfies the weak extension property. It is easily seen that this is the case if and only if the set of all $n$-tuples which generate a relatively compact subgroup is dense in $\Iso(\bG)^n$ for all $n$. We already do not know whether this is true when $n=1$.

Let us conclude these notes by stating a criterion for a Polish group to be L\'evy - this consists simply in isolating the ideas used by Pestov in \cite{Pestov2006a} to show that the isometry group of Urysohn's universal metric space is a L\'evy group. The proof is a straightforward adaptation of Pestov's proof as presented in 
\cite{Pestov2006a}, so we content ourselves 
with stating the criterion below without proof.

\begin{theorem}\label{Levy}
Let $\mcK$ be a metric Fra\"iss\'e class with limit $\bK$. Assume that:
\begin{itemize}
\item $\mcK$ has the extension property.
\item $\mcK$ has the $\ell_1$ property.
\item For any finite substructure $\bA$ of $\bK$, and any finite group $H$ acting on $\bA$ by automorphisms, the action $H \actson \bA$ extends to an action $H \actson \bK$ by automorphisms.
\end{itemize}
Then $G=\Aut(\bK)$, endowed with its usual Polish topology, is a L\'evy group.

Actually, there is a increasing chain of \emph{finite} subgroups which concentrates (for the normalized counting measure) and whose union is dense in $G$.
\end{theorem}

\bibliography{mybiblio}

\end{document}